\documentclass[]{article}

\usepackage{amssymb,amsmath,amsthm,array}
\usepackage{dsfont,graphicx}
\usepackage{pst-plot,pstricks,multido}

\newtheorem{thm}{Theorem}[section]
\newtheorem{cor}[thm]{Corollary}
\newtheorem{lemma}[thm]{Lemma}
\newtheorem{prop}[thm]{Proposition}
\theoremstyle{definition}
\newtheorem{defn}[thm]{Definition}
\newtheorem{rem}[thm]{Remark}
\newtheorem{ex}[thm]{Example}
\numberwithin{equation}{section}

\DeclareMathOperator{\supp}{supp}
\DeclareMathOperator{\Sp}{span}
\DeclareMathOperator{\Inv}{Inv}

\begin{document}

\setcounter{page}{1}

\begin{center} \Large\textbf{On a generalized $*$-product for copulas}\normalsize \end{center}

\vspace{0.5cm}
\begin{center} Pongpol Ruankong and Songkiat Sumetkijakan\footnote[1]{Corresponding author} \end{center}
\vspace{0.5cm}

\begin{abstract}\footnotesize
\noindent This paper focuses on a generalization of the $*$-product called $\mathbf{C}$-product. This product, first introduced by Durante, Klement and Quesada-Molina, was used to characterize classes of compatible copulas. The $\mathbf{C}$-product of copulas $A$ and $B$ is defined to be an integral of a function which involves the copulas $A$ and $B$ and the family of copulas $\mathbf{C}$. However, measurability of the integrand in the definition is questionable. We will discuss this in details and attempt to re-define the product. Then we derive some properties of the re-defined product.
\vskip 0.2cm
\normalsize
\end{abstract}
\vskip 0.2cm
\noindent Mathematics Subject Classification: 28A35\\
\noindent Keywords: copula, $*$-product, $\mathbf{C}$-product





\section{Introduction}\label{intro}

In 1959, the notion of copulas was introduced by Sklar. A copula links joint distribution to its marginals via the following equation where the existence of $C$ is guaranteed for all random variables $X,Y$. \begin{displaymath} F_{XY}(u,v)=C(F_X(u),F_Y(v)). \end{displaymath} Moreover, the copula $C$ is unique if $X$ and $Y$ are continuous random variables. A few decades later, in 1992, Darsow, Nguyen and Olsen \cite{Da} introduced a bilinear operation on the set of copulas called the $*$-product. The product was first invented for the purpose of studying Markov processes via copulas. This product proves to be useful, in general, as an operation on the set of copulas.

In this paper, we focus on one of its generalizations known as $\mathbf{C}$-product, which was introduced by Durante, Klement and Quesada-Molina \cite{Fre} for the purpose of studying Fr\'{e}chet classes. Because it was introduced only recently, not many researches have been done. So, we have very little knowledge on the product. Throughout this paper, we will call it $*_{\mathbf{C}}$ product to emphasize the link it has with the classical $*$-product. Given a family of copulas $\mathbf{C}=\{C_t\}_{t \in [0,1]}$, the $*_\mathbf{C}$ product is defined on the set of copulas as follows:
\begin{displaymath} (A*_{\mathbf{C}}B)(u,v)=\int_0^1{C_t(\partial_2A(x,t),\partial_1B(t,y))}~dt. \end{displaymath} However, it is questionable whether this product is well-defined because of the measurability of the integrand. In fact, it is reasonably easy to construct an example where the integrand is not Lebesgue measurable. The main purpose of this work is to limit the scope of the definition so that the product is well-defined. In order to do that, we restrict our attention to some reasonably large classes of families of copulas which make the integrand Lebesgue measurable for all copulas $A,B$ and for all $x,y \in [0,1]$.

This paper is organized as follows. In Section \ref{prelim}, important definitions and results invloving our work are given. We then discuss about measurability of the integrand in Section \ref{measurable}. In Section \ref{products}, we re-define the product and derive some of its properties. Finally, in Section \ref{one-zero}, we find the identity and zero of $(\mathfrak{C},*_{\mathbf{C}})$.

\section{Preliminaries}\label{prelim}

In this paper, we are only interested in 2-copulas and 3-copulas. Their definitions are given below. More details on copulas can be found in the classic book \cite{Ne} by Nelsen.

\begin{defn} A \emph{2-copula}, or simply \emph{copula}, is a function $C \colon [0,1]^2 \rightarrow [0,1]$ satisfying the conditions:
\begin{enumerate}
	\item $C(u,0)=C(0,v)=0$ for all $u,v \in [0,1]$.
	\item $C(u,1)=u$ and $C(1,v)=v$ for all $u,v \in [0,1]$.
	\item $C$ is $2$-increasing, i.e., for all $[u_1,u_2]\times[v_1,v_2]\subseteq [0,1]^2$, we have \begin{displaymath} C(u_2,v_2)-C(u_2,v_1)-C(u_1,v_2)+C(u_1,v_1) \ge 0. \end{displaymath}
\end{enumerate}
\end{defn}

\begin{defn} \emph{A 3-copula} is a function $C \colon [0,1]^3 \rightarrow [0,1]$ satisfying the conditions:
\begin{enumerate}
	\item $C(u,v,0)=C(u,0,w)=C(0,v,w)=0$ for all $u,v,w \in [0,1]$.
	\item $C(u,1,1)=u$, $C(1,v,1)=v$ and $C(1,1,w) =w$ for all $u,v,w \in [0,1]$.
	\item $C$ is $3$-increasing, i.e., for all cubes $[u_1,u_2]\times[v_1,v_2]\times[w_1,w_2] \in [0,1]^3$, we have
	\begin{align*}
	&C(u_2,v_2,w_2)-C(u_1,v_2,w_2)-C(u_2,v_1,w_2)-C(u_2,v_2,w_1)+\\
	&C(u_2,v_1,w_1)+C(u_1,v_2,w_1)+C(u_1,v_1,w_2)-C(u_1,v_1,w_1) \ge 0
	\end{align*}
\end{enumerate}
\end{defn}

Let $n$ be either $2$ or $3$. Then, according to Sklar's theorem (see, e.g., \cite{Ne}), for any random vector $(X_1,X_2,\dots,X_n)$, there exists an $n$-copula $C$ which links the joint distribution to its marginals as follows: \begin{displaymath} F_{X_1,\dots,X_n}(u_1,\dots,u_n)=C(F_{X_1}(u_1),\dots,F_{X_n}(u_n)). \end{displaymath} If the $X_i$'s are continuous random variables, then the $n$-copula $C$ is unique. Every $n$-copula is Lipschitz continuous with Lipschitz constant 1, as a consequence, its first-order partial derivatives exist almost everywhere. Moreover, they are bounded between $0$ and $1$, wherever exist.

Throughout this paper, the set of copulas is denoted by $\mathfrak{C}$. In addition, each copula induces a measure on the Borel subsets of $[0,1]^2$ as follows.

\begin{defn} Given a copula $C$, define a set function $\mu_C$ on the set of rectangles $[x_1,x_2]\times[y_1,y_2]$ in $[0,1]^2$ by \begin{displaymath} \mu_C([x_1,x_2]\times[y_1,y_2])=C(x_2,y_2)-C(x_2,y_1)-C(x_1,y_2)+C(x_1,y_1) \ge 0.
\end{displaymath}
By Caratheodory Extension Theorem, the set function $\mu_C$ can be extended to a measure on the Borel $\sigma$-algebra on $[0,1]^2$. Moreover, the measure $\mu_C$ is said to be doubly stochastic, i.e., it satisfies $\mu_C(B \times [0,1])=\mu_C([0,1]\times B)= \lambda(B)$ for every Borel set $B \subseteq [0,1]$, where $\lambda$ denotes Lebesgue measure. This measure is sometimes called \emph{$C$-measure}, \emph{$C$-volume} or \emph{mass of copula C}. 
\end{defn}

\begin{defn} The \emph{support} of a copula $C$, denoted by $\supp{C}$, is defined to be the complement of the union of all open subsets of $[0,1]^2$ with zero $C$-volume.
\end{defn}

Theoretically, the three most important copulas are the Fr\'{e}chet-Hoeffding upper and lower bounds and the product copula given by \begin{center}$M(u,v)=\min(u,v), W(u,v)=\max(u+v-1,0)$ and $\Pi(u,v)=uv$,\end{center} respectively. They represent comonotonicity, countermonotonicity and independence, respectively, between the two random variables.

\begin{ex} It can be shown that $\supp{M}$ is the main diagonal from $(0,0)$ to $(1,1)$, $\supp{W}$ is the other diagonal and $\Pi$ has full support, i.e., $\supp{\Pi}=[0,1]^2$.
\end{ex}

In their study of Markov processes, Darsow, Nguyen and Olsen \cite[p.~604]{Da} introduces a binary operation $\ast \colon \mathfrak{C} \times \mathfrak{C} \rightarrow \mathfrak{C}$ defined by \begin{displaymath}(A\ast B)(u,v) = \displaystyle\int\limits_0^1{\partial_2A(u,t)\partial_1B(t,v)}~dt, \end{displaymath} where $\partial_i$ denotes the first-order partial derivative with respect to the $i$-th variable. This operation is bilinear and is called the \emph{$\ast$-product} on $\mathfrak{C}$. Remark that it can be naturally extended to a bilinear operation on $\Sp\mathfrak{C}$. From straightforward computations, for any $C \in \mathfrak{C}$, we have the following identities:
\begin{align*}
M \ast C &= C \ast M = C,\\
\Pi \ast C &= C \ast \Pi = \Pi.
\end{align*}

Copulas $M$ and $\Pi$ can be viewed as the identity and the zero of $(\mathfrak{C},\ast)$, respectively. Moreover, denoted by $C^T$, the transpose of $C$, defined by $C^T(u,v)=C(v,u)$ is also a copula. A copula $B$ is said to be \emph{invertible} if there exists a copula $C$ such that $B \ast C = C \ast B = M$. The set of invertible copulas is denoted by $\Inv\mathfrak{C}$.

\begin{rem}
If they exist, left and right inverses of a copula $C\in \mathfrak{C}$ are unique and given by the transposed copula $C^T$ (for a proof, see \cite[Theorem 7.1]{Da}). 
\end{rem}

An important class of invertible copulas is the class of shuffles of $M$. This class attracts our interest because it is easy to compute. Moreover, the class of shuffles of $M$ is dense in $\mathfrak{C}$ with respect to the uniform norm. A definition of a shuffle of $M$ is given below.

\begin{defn} A copula $C$ is a \emph{shuffle of $M$} if and only if there exist a positive integer $n$, partitions $0 = s_0 < s_1 < \dots < s_n =1$ and $0 = t_0 < t_1 < \dots < t_n =1$ of $[0,1]$, and a permutation $\sigma$ on the set $\{1,2,\dots,n\}$ such that each $(s_{i-1},s_i)\times(t_{\sigma(i)-1},t_{\sigma(i)})$ is a square of $C$-volume $s_i-s_{i-1}$ and its intersection with the support of $C$ is one of the diagonals of the square.
\end{defn}

\begin{figure}[ht]
\psset{xunit=3cm,yunit=3cm}
\begin{center}
\begin{pspicture*}(-0.2,-0.2)(1.2,1.2)
\psline[linecolor=black](0,1)(1,1)
\psline[linecolor=black](1,0)(1,1)
\psline[linecolor=black](0,0)(0,1)
\psline[linecolor=black](0,0)(1,0)
\psline[linecolor=black,linestyle=dotted](0.2,0)(0.2,1)
\psline[linecolor=black,linestyle=dotted](0.7,0)(0.7,1)
\psline[linecolor=black,linestyle=dotted](0,0.5)(1,0.5)
\psline[linecolor=black,linestyle=dotted](0,0.8)(1,0.8)
  \psplot[linecolor=blue,plotpoints=400]{0}{0.2}{x 0.8 add}
  \psplot[linecolor=blue,plotpoints=400]{0.2}{0.7}{0.7 x sub}
  \psplot[linecolor=blue,plotpoints=400]{0.7}{1}{x 0.7 sub 0.5 add}
\uput{0.2}[270](0,0){$s_0$}
\uput{0.2}[270](1,0){$s_3$}
\uput{0.2}[270](0.2,0){$s_1$}
\uput{0.2}[270](0.7,0){$s_2$}
\uput{0.2}[180](0,0){$t_0$}
\uput{0.2}[180](0,0.5){$t_1$}
\uput{0.2}[180](0,0.8){$t_2$}
\uput{0.2}[180](0,1){$t_3$}

\end{pspicture*}
\end{center}
  \caption[ ]{the support of a shuffle of $M$ where $\sigma=(1~3~2)$} 
  \label{fig:shuffle}
\end{figure}
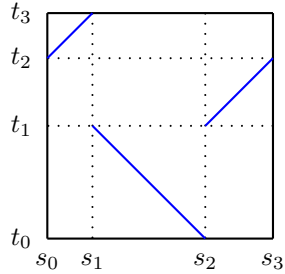 

An example of shuffles of $M$ is given below. For more details on shuffles of $M$, see, e.g., \cite{Du,Sa}.

\begin{ex} The \emph{straight shuffle of $M$} at $\alpha \in [0,1]$, denoted by $S_{\alpha}$, is defined to be the shuffle of $M$ supported on the straight line joining the points $(0,\alpha)$ and $(1-\alpha,1)$ and the straight line joining the points $(1-\alpha,0)$ and $(1,\alpha)$. Moreover, from straightforward computation, we have
\begin{displaymath}(S_{\alpha}*C)(u,v)=
\left\{
	\begin{array}{ll}
		C(u+1-\alpha,v)-C(1-\alpha,v) & \text{if}~~ 0 \le u \le \alpha \le 1\\
		v -C(1-\alpha,v)+C(u-\alpha,v) & \text{if}~~ 0 \le \alpha \le u \le 1.
	\end{array}
\right.\end{displaymath}

\label{straight}
\end{ex}

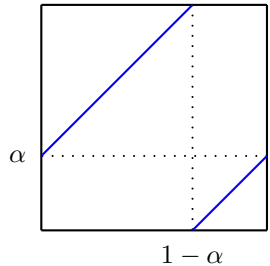
\begin{figure}[ht]
\psset{xunit=3cm,yunit=3cm}
\begin{center}
\begin{pspicture*}(-0.2,-0.2)(1.2,1.2)
\psline[linecolor=black](0,1)(1,1)
\psline[linecolor=black](1,0)(1,1)
\psline[linecolor=black](0,0)(0,1)
\psline[linecolor=black](0,0)(1,0)
\psline[linecolor=black,linestyle=dotted](0.67,0)(0.67,1)
\psline[linecolor=black,linestyle=dotted](0,0.33)(1,0.33)
  \psplot[linecolor=blue,plotpoints=400]{0}{0.67}{x 0.33 add}
  \psplot[linecolor=blue,plotpoints=400]{0.67}{1}{x 0.67 sub}
\uput{0.2}[270](0.67,0){$1-\alpha$}
\uput{0.2}[180](0,0.33){$\alpha$}
\end{pspicture*}
\end{center}
  \caption[ ]{the support of the straight shuffle of $M$ at $\alpha \in [0,1]$} 
  \label{fig:straight}
\end{figure} 

Now, let us recall a generalization of the $*$-product. The motivation behind this generalization comes from a research on compatibility of copulas. Copulas $C_{12},C_{13}$ and $C_{23}$ are said to be \emph{compatible} if there exists a 3-copula $\tilde{C}$ such that
\begin{align*}
C_{12}(u,v) &= \tilde{C}(u,v,1),\\
C_{13}(u,w) &= \tilde{C}(u,1,w),\\
C_{23}(v,w) &= \tilde{C}(1,v,w).
\end{align*}

Given copulas $A$ and $B$, the class $\mathcal{C}(A,B)$ is the set of all copulas that are compatible with copulas $A$ and $B$. In order to characterize these classes, Durante et al. \cite{Fre} defined the $*_{\mathbf{C}}$ product introduced in Section \ref{intro}.

\begin{prop}[\cite{Fre}, Theorem 4.1] Let $A,B \in \mathfrak{C}$. A copula $U$ is in the class $\mathcal{C}(A,B)$ if and only if there exists a family of copulas $\mathbf{C}=\{C_t\}_{t\in[0,1]}$ such that $U=A*_{\mathbf{C}}B$.
\end{prop}

\begin{ex}[\cite{Rem}, p. 237] For every $B \in \mathfrak{C}$ and every family of copulas $\mathbf{C}=\{C_t\}_{t \in [0,1]}$ such that $*_{\mathbf{C}}$ is well-defined, we have
\begin{align*}
B*_{\mathbf{C}}M &= B = M*_{\mathbf{C}}B,\\
(B*_{\mathbf{C}}W)(u,v) &= u - B(u,1-v),\\
(W*_{\mathbf{C}}B)(u,v) &= v- B(1-u,v).
\end{align*}
\end{ex}

\section{Measurability of the Integrand}\label{measurable}

In this section, we introduce various sets of conditions on the family of copulas so that the $*_{\mathbf{C}}$ product is well-defined. But first, let us give an example where the $*_{\mathbf{C}}$ product is not well-defined.

\begin{ex}

Let $P$ be a Lebesgue nonmeasurable subset of $[0,1].$ Consider the family of copulas $\mathbf{C}=\{C_t\}_{t \in [0,1]}$ where 
\begin{displaymath}C_t=
\left\{
	\begin{array}{ll}
		M & \text{if}~~ t \in P\\
		W & \text{if}~~ t \notin P.
	\end{array}
\right.\end{displaymath}

Then we can see that $C_t(\partial_2A(x,t),\partial_1B(t,y))$ is not Lebesgue measurable in the variable $t$ for some $A,B \in \mathfrak{C}$ and $x,y \in [0,1]$, e.g., $A,B=\Pi$ and any $x,y \in (0,1)$.
\end{ex}

From the above example, it is evident that the $*_{\mathbf{C}}$ product is not always well-defined since the integrand may not be a Lebesgue measurable function. We now give two sets of conditions such that $C_t(\partial_2A(x,t),\partial_1B(t,y))$ is a Lebesgue measurable function in the variable $t$.

\begin{thm} Let $\mathbf{C}=\{C_t\}_{t \in [0,1]}$ be a family of copulas which satisfies the following conditions:

\begin{enumerate}
\item $\mathbf{C}$ consists of countably many distinct copulas and
\item for each $A \in \mathfrak{C}$, the set $\{t \in [0,1] \colon C_t = A\}$ is Borel measurable.
\end{enumerate}
Then, for all $x,y \in [0,1]$ and for all $A,B \in \mathfrak{C}$, $C_t(\partial_2A(x,t),\partial_1B(t,y))$ is Lebesgue measurable in the variable $t$.
\label{mc}
\end{thm}
\begin{proof}Let $\mathbf{C} = \{C_t\}_{t \in [0,1]}$ be a family of copulas satisfying the two conditions. Since there are countably many distinct copulas. Let $E = \{C_1,C_2,\dots\}$ be an enumeration of the distinct copulas in the family.

For each $C_n \in E$, let $T_n=\{t \in [0,1] \colon C_t = C_n\}$. Then we can write $C_t(\partial_2A(x,t),\partial_1B(t,y))$ as \begin{displaymath} \sum_{n=1}^{\infty} \chi_{T_n}(t)C_n(\partial_2A(x,t),\partial_1B(t,y)), \end{displaymath}

\noindent which is a countable sum of Lebesgue measurable functions. Hence, it is Lebesgue measurable.
\end{proof}

Observe that the proof of the above theorem works perfectly fine if we replace Borel measurability by Lebesgue measurability.

\begin{thm} If the map $(t,x,y) \mapsto C_t(x,y)$ is Borel measurable, then for all $x,y \in [0,1]$ and for all $A,B \in \mathfrak{C}$, $C_t(\partial_2A(x,t),\partial_1B(t,y))$ is Lebesgue measurable in the variable $t$.
\label{mu}
\end{thm}

\begin{proof} For all $x,y \in [0,1]$ and $A,B \in \mathfrak{C}$, the map $t \mapsto (t,\partial_2A(x,t),\partial_1B(t,y))$ is Lebesgue measurable since each component function is Lebesgue measurable in the variable $t$. Then, being the composition of a Lebesgue measurable map $t \mapsto (t,\partial_2A(x,t),\partial_1B(t,y))$ and a Borel measurable map $(t,x,y) \mapsto C_t(x,y)$, the map $t \mapsto C_t(\partial_2A(x,t),\partial_1B(t,y))$ is Lebesgue measurable.
\end{proof}

Denoted by $\mathcal{M}_c$ the collection of families which satisfy the set of conditions in Theorem \ref{mc}, $\mathcal{M}_u$ the collection of families which satisfy the condition in Theorem \ref{mu} and $\mathcal{M}$ the collection of families such that, for all $A,B \in \mathfrak{C}$ and $x,y \in [0,1]$, the function $C_t(\partial_2A(x,t),\partial_1B(t,y))$ is Lebesgue measurable in the variable $t$. We have just shown that $\mathcal{M}_c$ and $\mathcal{M}_u$ are subcollections of $\mathcal{M}$. In practice, it is not easy to determine whether a family $\mathbf{C}$ is an element of $\mathcal{M}$. That is why we introduce the collections $\mathcal{M}_c$ and $\mathcal{M}_u$.

\begin{lemma} If a family of copulas satifies the set of conditions in Theorem \ref{mc}, then it also satisfies the condition in Theorem \ref{mu}. That is, $\mathcal{M}_c \subseteq \mathcal{M}_u$. 
\end{lemma}

\begin{proof} Let $\mathbf{C}=\{C_t\}_{t \in [0,1]} \in \mathcal{M}_c$ be a family of copulas. For each $C_n \in \mathbf{C}$, let $T_n=\{t \in [0,1] \colon C_t = C_n\}$. Then we can write \begin{center}$\displaystyle C_t(x,y)= \sum_{n=1}^{\infty} \chi_{T_n}(t)C_n(x,y).$\end{center} For any $a \in [0,1]$, the inverse image of the interval $[0,a]$ under the map $(t,x,y) \mapsto C_t(x,y)$ is equal to $\displaystyle\bigcup_{n=1}^{\infty}T_n \times C_n^{-1}([0,a])$. Observe that $T_n$ and $C_n^{-1}([0,a])$ are Borel measurable. Hence, the inverse image of the interval $[0,a]$ under the map $(t,x,y) \mapsto C_t(x,y)$ is Borel measurable.
\end{proof}

Though $\mathcal{M}_u$ is quite a large subcollection of $\mathcal{M}$, we are still not able to fully characterize all elements of $\mathcal{M}$. The following proposition helps us in dealing with families which behave well outside a set of Lebesgue measure zero.

\begin{prop} Let $\mathbf{C} \in \mathcal{M}$. If $\mathbf{D}$ is another family of copulas such that $D_t=C_t$ a.e. $t \in [0,1]$, then $\mathbf{D} \in \mathcal{M}$ and the products $*_{\mathbf{C}}$ and $*_{\mathbf{D}}$ are identical. We say that the family $\mathbf{D}$ is $*$-equivalent to the family $\mathbf{C}$.
\label{ae}
\end{prop}

\begin{proof} The result readily follows from the fact that if $f=g$ a.e. and $f$ is Lebesgue measurable, then $g$ is also Lebesgue measurable.
\end{proof}


\section{The $*_{\mathbf{C}}$ Product}\label{products}

In this section, we properly re-define the $*_{\mathbf{C}}$ product. Then we derive some of its properties.

\begin{defn} Let $\mathbf{C} \in \mathcal{M}$. The $*_{\mathbf{C}}$ product of copulas $A$ and $B$ is defined by \begin{displaymath} (A*_{\mathbf{C}}B)(x,y) = \int_0^1{C_t(\partial_2A(x,t),\partial_1B(t,y))}~dt. \end{displaymath}
\end{defn}

\begin{rem}[\cite{Fre}, Proposition 3.1] For all $\mathbf{C} \in \mathcal{M}$ and for all $A,B \in \mathfrak{C}$, we have $A*_{\mathbf{C}}B \in \mathfrak{C}$.
\end{rem}

\begin{lemma} Let $\mathbf{C} \in \mathcal{M}$ and $A,B \in \mathfrak{C}$. If $A$ is right invertible or $B$ is left invertible with respect to the $*$-product, then $A*_{\mathbf{C}}B=A*B$.
\label{smallest-class}
\end{lemma}

\begin{proof} It suffices to prove only the statement for $A$ as the other can be proved analogously. Let $A$ be a right invertible copula. Then, we have $\partial_2{A}(u,v)\in \{0,1\}$ almost everywhere. Let $Z$ be the set $\{(u,v)\in [0,1]^2 \colon \partial_2{A}(u,v) =1\}$. Compute
\begin{align*}
(A*_{\mathbf{C}}B)(x,y) &= \int_0^1{C_t(\partial_2A(x,t),\partial_1B(t,y))}~dt\\
&=\int_Z{C_t(1,\partial_1B(t,y))}~dt\\
&=\int_Z{\partial_1B(t,y)}~dt\\
&=\int_0^1{\partial_2A(x,t)\partial_1B(t,y)}~dt\\
&=(A*B)(x,y).
\end{align*}
\end{proof}

\begin{lemma} The collection $\{\Pi*_{\mathbf{C}}\Pi \colon \mathbf{C} \in \mathcal{M}\}=\mathfrak{C}$.
\label{pipi}
\end{lemma}

\begin{proof} For any copula $C \in \mathfrak{C}$, consider the family $\mathbf{C}$ consisting of $C_t=C$ for all $t \in [0,1]$. Then we have that $\mathbf{C} \in \mathcal{M}$ and 
\begin{displaymath}
(\Pi*_{\mathbf{C}}\Pi)(x,y)=\int_0^1{C(x,y)}~dt =C(x,y).
\end{displaymath}
This completes the proof.
\end{proof}

\begin{thm} If $\mathbf{C}_n, \mathbf{C} \in \mathcal{M}$ such that $C_{n,t}(x,y) \rightarrow C_t(x,y)$ pointwise for all $t \in [0,1]$, then $(A*_{\mathbf{C}_n}B)(x,y) \rightarrow (A*_{\mathbf{C}}B)(x,y)$ pointwise.
\end{thm}

\begin{proof} Observe that, for a fixed $t\in [0,1]$, we have \begin{displaymath}C_{n,t}(\partial_2A(x,t),\partial_1B(t,y)) \rightarrow C_t(\partial_2A(x,t),\partial_1B(t,y))\end{displaymath} pointwise.
Moreover, $C_{n,t}(\partial_2A(x,t),\partial_1B(t,y))$ and $C_t(\partial_2A(x,t),\partial_1B(t,y))$ are bounded by $1$ which is Lebesgue integrable on $[0,1]$. By Dominated Convergence Theorem, we have that
\begin{displaymath} \int_0^1 C_{n,t}(\partial_2A(x,t),\partial_1B(t,y))~dt \rightarrow \int_0^1 C_t(\partial_2A(x,t),\partial_1B(t,y))~dt \end{displaymath}
pointwise, which completes the proof.
\end{proof}

\begin{cor} If $\mathbf{C}_n, \mathbf{C} \in \mathcal{M}$ such that $C_{n,t} \rightarrow C_t$ uniformly for all $t \in [0,1]$, then $A*_{\mathbf{C}_n}B \rightarrow A*_{\mathbf{C}}B$ uniformly.
\end{cor}

\begin{ex} Recall that the set of shuffles of $M$ is dense in $\mathfrak{C}$ with respect to the uniform norm. Hence, given a family $\mathbf{C}=\{C_t\}_{t \in [0,1]}$, we can find families of shuffles of $M$, $\mathbf{S}_n= \{S_{n,t}\}_{t \in [0,1]}$, such that $A*_{\mathbf{S}_n}B \rightarrow A*_{\mathbf{C}}B$ uniformly.
\end{ex}

Our motivation for the previous example is the computation of $A*_{\mathbf{C}}B$. One can see that given a family $\mathbf{C}=\{C_t\}_{t \in [0,1]}$, it is not easy to obtain an explicit formula for $A*_{\mathbf{C}}B$. But with the above result, the computation seems more feasible.


\section{Identity and Zero of $(\mathfrak{C},*_{\mathbf{C}})$}\label{one-zero}

Recall that the identity and the zero of $(\mathfrak{C},*)$ are $M$ and $\Pi$, respectively. 

\begin{thm} For all $\mathbf{C} \in \mathcal{M}$, the identity of $(\mathfrak{C},*_{\mathbf{C}})$ exists and is unique. Moreover, it is the Fr\'{e}chet-Hoeffding upper bound $M$.
\end{thm}

\begin{proof} Let $\mathbf{C} \in \mathcal{M}$. Since $M$ is invertible, from Lemma \ref{smallest-class}, we have \begin{displaymath} M*_{\mathbf{C}}A=M*A=A=A*M=A*_{\mathbf{C}}M \end{displaymath} for all $A \in \mathfrak{C}$. For the uniqueness, suppose $M'$ is another identity. Then we have \begin{displaymath} M=M*_{\mathbf{C}}M'=M'. \end{displaymath} Hence, for any $\mathbf{C} \in \mathcal{M}$, the copula $M$ is the identity for the $*_{\mathbf{C}}$ product.
\end{proof}

\begin{thm} Let $\mathbf{C} \in \mathcal{M}$. The zero of $(\mathfrak{C},*_{\mathbf{C}})$, if exists, is unique and is the product copula $\Pi$.
\end{thm}

\begin{proof} The uniqueness part is easy. Let $U,V$ be zeroes for the $*_{\mathbf{C}}$ product. Then $U=U*_{\mathbf{C}}V=V$. Now, to see that $\Pi$ is the zero, if exists, it requires some work.

Let $U$ be the zero for the $*_{\mathbf{C}}$ product. For each $S_{\alpha}$, the straight shuffle of $M$ at $\alpha \in [0,1]$, since $S_{\alpha}$ is invertible, we have $S_{\alpha}*U= S_{\alpha}*_{\mathbf{C}}U=U$.

Recall the formula
\begin{displaymath}(S_{\alpha}*C)(x,y)=
\left\{
	\begin{array}{ll}
		C(x+1-\alpha,y)-C(1-\alpha,y) & \text{if}~~ 0 \le x \le \alpha \le 1\\
		y -C(1-\alpha,y)+C(x-\alpha,y) & \text{if}~~ 0 \le \alpha \le x \le 1.
	\end{array}
\right.\end{displaymath}
\noindent Then copula $U$ must satisfy the two functional equations
\begin{align}
U(x+1-\alpha,y)&=U(x,y)+U(1-\alpha,y)\quad\text{if}~ 0 \le x \le \alpha \le 1~ \text{and}\label{fe-1}\\
U(x-\alpha,y)+y&=U(x,y)+U(1-\alpha,y)\quad\text{if}~ 0 \le \alpha \le x \le 1.
\label{fe-2}
\end{align}

We will solve the above equations and show that the only copula which satisfies them is the product copula $\Pi$.

Fix $y \in [0,1]$ and let $f(x)=U(x,y)$. Then, from the properties of copulas, $f$ is a continuous mapping on $[0,1]$ with boundary contitions $f(0)=0$ and $f(1)=y$.
Then \eqref{fe-1} and \eqref{fe-2} become
\begin{align}
f(x+1-\alpha)&=f(x)+f(1-\alpha) \quad\text{if}~ 0 \le x \le \alpha \le 1~ \text{and} \label{fe-3}\\
f(x-\alpha)+f(1)&=f(x)+f(1-\alpha) \quad\text{if}~ 0 \le \alpha \le x \le 1. \label{fe-4}
\end{align}

First, we solve \eqref{fe-3}. Let $z=1-\alpha$. Then \eqref{fe-3} becomes the well-known Cauchy equation
\begin{displaymath} f(x+z)=f(x)+f(z)\end{displaymath} where $0\le x \le 1$, $0 \le z \le 1$ and $0 \le x+z \le 1$. Recall that we are solving for a function $f \colon [0,1] \rightarrow [0,1]$. The case where $f$ is a real-valued function and the case where $f$ is a rational-valued function are well-known.

Observe that $f(\frac{m}{n})= f(\frac{m-1}{n})+f(\frac{1}{n})$ for all $m,n \in \mathbb{N}$ such that $1 \le m \le n$. Hence, by induction, we have $f(\frac{m}{n})=mf(\frac{1}{n}).$ Thus $f(1)=nf(\frac{1}{n})$. In other words, $f(\frac{1}{n})= \frac{1}{n}f(1)$ for all $n \in \mathbb{N}$. Therefore $f(\frac{m}{n})=\frac{m}{n}f(1)$ for all $m,n \in \mathbb{N}$ such that $1 \le m \le n$, i.e. $f(r)=rf(1)$ for all $r \in \mathbb{Q}\cap(0,1].$ We know that $f$ is continuous and $f(0)=0$. Hence, for all $x \in [0,1]$, we have \begin{equation} f(x)=xf(1). \label{sol-1}\end{equation}

Now, we solve \eqref{fe-4}. Observe that $f(x-x)+f(1)=f(x)+f(1-x)$. Hence, $f(1-x)=f(1)-f(x)$ for all $x \in [0,1].$ Thus, from \eqref{fe-4}, we have $f(x-\alpha)=f(x)-f(\alpha)$ for all $0 \le \alpha \le x \le 1$. In other words, $f(x)=f(x-\alpha)+f(\alpha)$ for all $0 \le \alpha \le x \le 1$. Again, we have $f(\frac{m}{n})= f(\frac{m-1}{n})+f(\frac{1}{n})$ for all $m,n \in \mathbb{N}$ such that $1 \le m \le n$. This is the same equation as the one we just solved. Hence, for all $x \in [0,1]$, we also have that \begin{equation} f(x)=xf(1). \label{sol-2}\end{equation}

From \eqref{sol-1} and \eqref{sol-2}, we have $U(x,y)=f(x)=xf(1)=xy$ for all $x,y \in [0,1]$. Thus, the only copula which satifies \eqref{fe-1} and \eqref{fe-2} is the product copula $\Pi$.
\end{proof} 

In the following lemma, we obtain a necessary condition for the $*_{\mathbf{C}}$ product to have a zero.

\begin{lemma} If $\mathbf{C} \in \mathcal{M}$ is a family such that $*_{\mathbf{C}}$ has a zero, then \begin{displaymath} \int_0^1C_t(x,y)~dt =xy \end{displaymath} for all $x,y \in [0,1]$.
\label{nc}
\end{lemma}

\begin{proof} If $*_{\mathbf{C}}$ has a zero, then $\Pi(x,y)=(\Pi*_{\mathbf{C}}\Pi)(x,y)=\displaystyle\int_0^1C_t(x,y)~dt$.
\end{proof}

Recall from Lemma \ref{pipi} that $\{\Pi*_{\mathbf{C}}\Pi \colon \mathbf{C} \in \mathcal{M}\} =\mathfrak{C}$. Hence, $\Pi*_\mathbf{C}\Pi$ can be any copula. But for the product $*_\mathbf{C}$ to have a zero, $\Pi*_\mathbf{C}\Pi$ can only be the product copula $\Pi$. One can see that, for the product $*_\mathbf{C}$ to have a zero, the underlying family $\mathbf{C}$ must be extremely special.

\begin{ex} Given a copula $C \in \mathfrak{C}$. Let $\mathbf{C}=\{C\}_{t\in[0,1]}$. If $C=\Pi$, then the $*_{\mathbf{C}}$ product is simply the classical $*$-product, which has a zero. If $C \neq \Pi$, then the $*_{\mathbf{C}}$ product has no zero by the above Lemma.
\end{ex}

\begin{ex} Let $\mathbf{C}$ be a family of copulas where $C_t = \Pi$ a.e. $t \in [0,1]$. Then, $*_{\mathbf{C}}$ has a zero since the families $\mathbf{C}$ and $\{\Pi\}_{t \in [0,1]}$ are $*$-equivalent. In fact, $*_{\mathbf{C}}$ is identical to the classical $*$-product.
\end{ex}

\begin{ex} Recall that the Farlie-Gumbel-Morgenstern(FGM) copulas are defined, for $\theta \in [-1,1]$, by $C_{\theta}(u,v)=uv+\theta uv(1-u)(1-v)$. Let $\mathbf{C}=\{C_t\}_{t \in [0,1]}$ where $C_t$ is equal to $C_{\theta}$ if $t \in [0,1/2]$ and is equal to $C_{-\theta}$ otherwise. It is easily seen that the family $\mathbf{C}$ satisfies the condition in Lemma \ref{nc}.
\label{ex-nc}
\end{ex}

We will show that the $*_{\mathbf{C}}$ product in the above example has no zero, which implies that the codition in Lemma \ref{nc} is not a sufficient condition for the product to have a zero.

\begin{ex} Consider a family of copulas $\mathbf{C}$ in Example \ref{ex-nc} where $\theta \neq 0$. Compute
\begin{align*}
(A*_{\mathbf{C}}\Pi)(x,y) &= \int_0^1{C_t(\partial_2 A(x,t),y)}~dt\\
&=  \int_0^{1/2}{C_{\theta}(\partial_2 A(x,t),y)}~dt+ \int_{1/2}^1{C_{-\theta}(\partial_2 A(x,t),y)}~dt\\
&= \int_0^{1/2}\partial_2 A(x,t)y + \theta \partial_2 A(x,t) y (1-\partial_2 A(x,t))(1-y)~dt+\\&~~~\int_{1/2}^1\partial_2 A(x,t)y + \theta \partial_2 A(x,t) y (1-\partial_2 A(x,t))(1-y)~dt\\
&= xy+ \theta y (1-y) \bigg[ \int_0^{1/2} \partial_2 A(x,t)(1-\partial_2 A(x,t))~dt- \\&~~~\int_{1/2}^1 \partial_2 A(x,t)(1-\partial_2 A(x,t))~dt \bigg].
\end{align*}

\noindent Choose $A = C_{\theta}$. From straightforward computation, if $x \notin \{0,1\}$, then \begin{displaymath} \int_0^{1/2} \partial_2 A(x,t)(1-\partial_2 A(x,t))~dt -\int_{1/2}^1 \partial_2 A(x,t)(1-\partial_2 A(x,t))~dt = \frac{x^2}{2\theta x(x-1)} \neq 0. \end{displaymath}

\noindent Thus $C_{\theta}*_{\mathbf{C}}\Pi \neq \Pi$. Therefore, $*_{\mathbf{C}}$ has no zero.

\end{ex}


\newpage


\begin{thebibliography}{99}













\bibitem{Da} W.F.~Darsow,  B.~Nguyen and E.T.~Olsen,  Copulas and Markov processes,
{\it Illinios J.~Math.}, {\bf 36}(1992), 600--642.

\bibitem{Norm} W.F.~Darsow and E.T.~Olsen,  Norms for Copulas,
{\it Int. J.~Math. Math.~Sci.} {\bf 18(3)}(1995), 417--436.

\bibitem{Fre} F.~Durante,  E.P.~Klement and J.J.~Quesada-Molina,  Copulas: compatibility and Fr\'{e}chet classes. [Online] Available:
\texttt{http://arxiv.org/abs/0711.2409v1},
(November 15, 2007).

\bibitem{Rem} F. Durante, E.P. Klement, J.J. Quesada-Molina, P. Sarkoci, Remarks on two product-like constructions for copulas, {\it Kybernetika (Prague)}, {\bf 43(2)}(2007), 235-244.

\bibitem{Du} F. Durante, P. Sarkoci, C. Sempi, Shuffles of copulas, {\it J. Math. Anal. Appl.}, {\bf 352}(2009), 914-921.

\bibitem{Ne} R.B.~Nelsen, \textit{An Introduction to Copulas}, (2nd ed.), Springer, 2006.

\bibitem{Sa} T.~Santiwipanont and S.~Sumetkijakan,  Mutual complete dependence copulas and the $*$-product, preprint.

\end{thebibliography}
\end{document}